\documentclass[a4paper,12pt]{amsart}
\usepackage[T1]{fontenc}

\usepackage{amssymb,amsthm,amsmath}
\usepackage[dvips]{graphicx}

\newtheorem{lemma}[equation]{Lemma}
\newtheorem{proposition}[equation]{Proposition}
\newtheorem{theorem}[equation]{Theorem}
\newtheorem{corollary}[equation]{Corollary}

\newtheorem{openproblem}[equation]{Open problem}
\newtheorem{fact}[equation]{Fact}
\newtheorem{conjecture}[equation]{Conjecture}

\theoremstyle{remark}

\numberwithin{equation}{section}

\DeclareMathOperator*{\arth}{ar\,tanh}
\newcommand{\R}{{\mathbb R}}
\newcommand{\Rn}{{\R}^n}

\newcommand{\Bn}{{\mathbb B}^n}
\newcommand{\HUP}{{\mathbb H}^2}
\renewcommand{\H}{{\mathbb H}}
\newcommand{\Hn}{{\mathbb H}^n}

\newcounter{minutes}
\divide\time by 60
\newcounter{hours}
\multiply\time by 60 \addtocounter{minutes}{-\time}

\begin{document}

\title[Balls in the triangular ratio metric]{\bf  Balls in the triangular ratio metric}

\author[S. Hokuni]{S. Hokuni}
\address{Department of Mathematics and Statistics, University of Turku,
FIN-20014 Turku, Finland}
\email{samhok@utu.fi}

\author[R. Kl\'en]{R. Kl\'en}
\address{Department of Mathematics and Statistics, University of Turku,
FIN-20014 Turku, Finland}
\email{riku.klen@utu.fi}

\author[Y. Li]{Y. Li}
\address{College of Science,
Central South University of
Forestry and Technology, Changsha,  Hunan 410004, People's Republic
of China}
\email{yaxiangli@163.com}

\author[M. Vuorinen]{M. Vuorinen}
\address{Department of Mathematics and Statistics, University of Turku,
FIN-20014 Turku, Finland}
\email{vuorinen@utu.fi}



\begin{abstract}
  We consider the triangular ratio metric and estimate the radius of convexity for balls in some special domains and prove the inclusion relations of metric balls defined by the triangular ratio metric, the quasihyperbolic metric and the $j$-metric.
\end{abstract}

\maketitle

\def\thefootnote{}
\footnotetext{
\texttt{\tiny File:~\jobname .tex,
          printed: \number\year-\number\month-\number\day,
          \thehours.\ifnum\theminutes<10{0}\fi\theminutes}
}
\makeatletter\def\thefootnote{\@arabic\c@footnote}\makeatother

\section{Introduction}

Geometric function theory makes use of several metrics for subdomains
of $\mathbb{R}^n$. It has turned out that sometimes hyperbolic metric or more generally, hyperbolic
type metrics, are more natural than  the Euclidean metric, because of their better invariance or quasiinvariance properties under well-known classes of mappings such as M\"obius transformations, bilipschitz
or quasiconformal maps. In the recent years many authors have contributed to this field, see e.g.
 \cite{ha, himps,hmm,  Kl3,  KV, KV2, l,mv, rt, rt2, v1, v2, w}. On the other hand, hyperbolic type metrics are sometimes difficult to estimate and it is desirable to find simple
expressions to serve as comparison functions. Two such expressions,
the visual angular metric and the triangular ratio metric, were recently studied in \cite{klvw}. Here our
goal is to  continue the study of the triangular ratio metric for proper subdomains of $\Rn$. In particular,
we study the local convexity of balls in this metric for some
simple domains. For some other metrics, results of this type were
recently proved by R. Kl\'en in \cite{k08,k10}, in answer to a question posed in \cite{vu07}.

We study also inclusion relations between balls in different metrics. Some of the metrics we study are the triangular ratio metric, the quasihyperbolic metric and the $j$-metric. We consider the inclusion relations in general domains as well as in some specific examples as the punctured space and the half-space. These kind of results for  hyperbolic type metrics have been studied in \cite{KV,KV2}.

For a domain $G \subsetneq \mathbb{R}^n$, and $x,y \in G$, we define the triangular ratio metric \cite{ba,ha} by
\begin{equation}\label{definiton}
  s_G(x,y)= \sup_{z \in \partial G} \frac{|x-y|}{|z-x|+|z-y|} \in [0,1],
\end{equation}
the $j$-metric \cite{gp,vu85} by
\[
  j_G(x,y) =
 \log \left( 1+\frac{|x-y|}{\min \{ d_G(x),d_G(y) \}} \right),\]
where $d_G(z) = d(z, \partial G)$,
and the quasihyperbolic metric \cite{gp} by
\[
  k_G(x,y) = \inf_\gamma \int_\gamma \frac{|dz|}{d_G(z)},
\]
where the infimum is taken over all rectifiable curves $\gamma \subset G$ joining $x$ and $y$. G.J. Martin and B.G. Osgood proved in \cite[page 38]{mo} the following formula for the quasihyperbolic distance:
\[
  k_G(x,y) = \sqrt{\alpha^2 + \log^2 \frac{|x|}{|y|}}, \quad x,y \in G = \Rn \setminus \{ 0 \},
\]
where $\alpha = \measuredangle (x,0,y)$. For a metric space $(G,m)$ we define the metric ball for $x \in G$ and $r>0$ by $B_m(x,r) = \{ y \in G \colon m(x,y) < r \}$.

In Sections 3-5 we consider local convexity properties of balls in triangular ratio metric. In Section 3 we consider the punctured space $\Rn \setminus \{ 0 \}$, in Section 4 the half-space $\Hn$ and in Section 5 the punctured half-space $\Hn \setminus \{ e_n \}$ and polygons $P \subset \R^2$. In Sections 6 and 7 we study the inclusion of balls defined by the triangular ratio metric, the quasihyperbolic metric and the $j$-metric. In Section 6 we consider these metrics in general domains and in Section 7 in the punctured space and in the half-space.

Our main results are the following two theorems.

\begin{theorem}\label{mainthm1}
  (1) Let $G = \mathbb{R}^n \setminus \{ z \}$, $z\in\mathbb{R}^n$, $x \in G$ and $r > 0$. The metric ball $B_s (x,r)$ is (strictly) convex if $r \leq 1/2$ $(r < 1/2)$.

  \noindent (2) Let $x \in G = \mathbb{H}^n$ and $r\in(0,1)$. Then
  \[
    B_s(x,r)=B^n \left( x-e_n x_n \left( 1-\frac{1+r^2}{1-r^2} \right) ,\frac{2x_n r}{1-r^2} \right)
  \]
  and $B_s(x,r)$ is thus strictly convex.

  \noindent (3) Let $x =(x_1,x_2) \in G= \HUP \setminus \{ e_n \}$ with $x_2 < |x_1|$ and $r \in (0,r_0]$, where
  \[
    r_0 = \frac{\sqrt{x_1^2+x_2^2}-\sqrt{2}x_2}{|x_1|+x_2}.
  \]
  Then $B_s(x,r)$ is convex.

  \noindent (4) Let $P \subset \mathbb{R}^2$ be a polygon and $x \in P$. Then $B_s(x,r)$ is convex for all $r \in (0,1/2]$. Moreover, if $P$ is convex then $B_s(x,r)$ is convex for all $r \in (0,1)$.
\end{theorem}

\begin{theorem}\label{mainthm2}
  (1) Suppose that $G\subset \mathbb{R}^n$ is a domain. For each $x\in G$ and $r\in (0,1)$, we have $$B^n \left( x, \frac{2r}{1+r}d_G(x)\right) \subset B_s(x, r)\subset B^n \left( x, \frac{2r}{1-r}d_G(x) \right) .$$

  \noindent (2) Suppose that $G\subset \mathbb{R}^n$ is a domain. For each $x\in G$ and $r\in (0,1)$, we have $$B_j(x, \log(1+2r))\subset B_s(x, r)$$ and the inclusion is sharp if there exists a point $w$ in $\partial B_s(x,r) $ such that $d_G(x)=d_G(w)$ and $$\partial G\cap S^{n-1}(w,d_G(w))\cap S^{n-1}(x,d_G(x))\neq \emptyset.$$ Moreover, for each $x\in G$ and $r\in (0,\frac{1}{3})$, we have $$B_s(x, r)\subset B_j \left( x, \log(1+\frac{2r}{1-3r}) \right). $$

\noindent (3) Let $G=\mathbb{R}^n\setminus \{0\}$. For each $x\in G$ and $r\in (0,1)$, we have $$B\left( x, \frac{2r}{1+r}|x| \right) \subset B_s(x, r)\subset B \left( x, \frac{2r}{1-r}|x| \right) .$$ Moreover, the inclusions are sharp.

\noindent (4) Let $G=\mathbb{R}^n\setminus \{0\}$. For each $x\in G$ and $r\in (0,1)$, we have $$B_j(x, m)\subset B_s(x, r)\subset B_j(x, M),$$ where $m=\log(1+2r)$, and $M=\log(1+\frac{2r}{1-r})$.  Moreover, the inclusions are sharp.

\end{theorem}

\section{Preliminary results}
\begin{fact}
\label{slopedef}
Let $r=r(\theta)$ be a function given in polar coordinates. Then the slope of the curve $r=r(\theta)$ at the point $(r,\theta)$ is
\begin{equation*}
\frac{r(\theta)+\tan(\theta)\frac{dr}{d\theta}}{-r(\theta)\tan(\theta)+\frac{dr}{d\theta}}.
\end{equation*}
\end{fact}

 The supremum in the definition \eqref{definiton}  of $s$ is attained at a point $z$ that is on a maximal ellipsoid with focii at $x$ and $y$. Therefore it is clear that $s$ is monotone with respect to the domain, which means that if $G \subset \mathbb{R}^n$ and $G' \subset G$ are domains and $x,y \in G'$ then $s_{G}(x,y) \le s_{G'}(x,y)$.

By the definition and the monotonicity with respect to domain it is easy to prove that for all $x \in G$ and $r \in (0,1)$
\begin{equation}\label{intersection}
  B_{s_G}(x,r) = \bigcap_{z \in \partial G} B_{s_{\mathbb{R}^n \setminus \{ z \}}}(x,r).
\end{equation}

We denote the hyperbolic distance by $\rho_G$, where $G$ is either the unit ball $\Bn$ or the upper half-space $\Hn$ (\cite{Beardon83}, \cite[pp.19-32]{Vu2}).

\begin{lemma}\label{le61}
Let $a,b>0$.\\
(1) The function $f_1(r)\equiv\frac{\log(1+a r)}{r}$ is strictly decreasing from $(0,\infty)$ onto $(0, a)$.\\
(2) The function $f_2(r)\equiv\frac{r}{2-r}-\arth\,r$ is strictly decreasing from $(0,1)$ onto $(-\infty, 0)$.\\
(3) The function $f_3(r)\equiv\frac{r}{r(r-2)\log(1-r)}$ is strictly decreasing from $(0,1)$ onto $(0, 1/2)$.\\
(4) The function $f_4(t)\equiv\log\,t-\frac{t-1}{t+1}$ is increasing from $(1,\infty)$ onto $(0, \infty)$.
\end{lemma}

\begin{proof}
(1) Let $f_{11}(r)=\log(1+a r)$ and $f_{12}(r)=r$. It is clear that $f_{11}(0^+)=f_{12}(0^+)=0$. By differentiation,
$$\frac{f'_{11}(r)}{f'_{12}(r)}=\frac{a}{1+ar}$$
which is strictly decreasing in $(0,\infty)$. Therefore, we get the monotonicity of $f_1$. The limiting values are clear by l'H\^opital's rule.

(2) By differentiation,
$$f'_2(r)=\frac{-3r^2+4r-2}{(1-r^2)(2-r)^2}$$
which is negative. Therefore, $f_2$ is strictly decreasing. The limiting values are clear.

(3)  Let $f_{31}(r)=\frac{r}{2-r}$ and $f_{32}(r)=-\log(1-r)$. It is clear that $f_{31}(0^+)=f_{32}(0^+)=0$. By differentiation,
$$\frac{f'_{31}(r)}{f'_{32}(r)}=\frac{2(1-r)}{(2-r)^2}\equiv 2\phi(r).$$
By differentiation, we have
$$\phi'(r)=-\frac{r}{(2-r)^3}<0.$$ Then $\phi$ is strictly decreasing in $(0,1)$. Therefore, we get the monotonicity of $f_3$. The limiting values are clear by l'H\^opital's rule.

(4) By differentiation,
$$f'_4(t)=\frac{t(1+t)^2}{t^2+1}>0,$$
Therefore, $f_4$ is strictly increasing. The limiting values are clear.
\end{proof}

\section{Convexity of balls in punctured space}

We consider the metric $s$ in the punctured space $\R^n \setminus \{ 0 \}$. We first compare $s$ with the hyperbolic metric $\rho$ and then study convexity of metric balls $B_s(x,r)$.

By the definition it is clear that for $x,y \in G = \R^n \setminus \{ 0 \}$, we have
 \[
  s_{G}(x,y) = \frac{|x-y|}{|x|+|y|}.
\]
This special case of the triangular ratio metric has been studied in \cite{AseevSychevTetenov05, ba}.

\begin{theorem}
The inequality $ s_{G}(x,y) \le \frac{1}{\log 3} j_{G}(x,y)$ holds for all $x\,,y\in G=\R^n \setminus \{ 0 \}$.
\end{theorem}

\begin{proof}Without loss of generality, we may assume that $|x|\leq |y|$. Then
for all $x\,,y\in G$, by Lemma \ref{le61} (1), we have
\begin{eqnarray*}
\frac{j_{G}(x,y)}{s_{G}(x,y)}=\frac{\log \left( 1+\frac{|x-y|}{|x|} \right) }{\frac{|x-y|}{|x|+|y|}}
\ge \log \left( 1+\frac{|x|+|y|}{|x|} \right)
\ge \log 3.
\end{eqnarray*}
Equality holds if $0,x,y$ are collinear and $|x|=|y|$.
\end{proof}

\begin{openproblem}We notice that $s_G(x,y)\leq \frac{1}{\log 3} j_G(x,y)<j_G(x,y)$ for $G=\mathbb{R}^n\setminus\{0\}$. A natural problem is that is it true that $s_G\leq j_G$ holds for all subdomains $G\subset \mathbb{R}^n$?
\end{openproblem}

We only get the following constant which is larger than $1$ but less than $2$. But we note that very recently the inequality  $s_G(x,y) \leq \frac{1}{\log 3}j_G(x,y)$ was proved in \cite{CHKV}.

\begin{theorem} Suppose that $G\subset \mathbb{R}^n$ is a domain. Then for all $x,y\in G$ we have $$s_G(x,y)\leq \frac{1}{\log2}j_G(x,y).$$
\end{theorem}
\begin{proof}
For given $x,y\in G$, we may assume that $d_G(x)\leq d_G(y)$. Then if $|x-y|> d_G(x)$, obviously we have $$s_G(x,y)\leq 1\leq \frac{1}{\log2}j_G(x,y).$$
If  $|x-y|\leq d_G(x)$, then $$s_G(x,y)\leq \frac{|x-y|}{2d_G(x)}\leq \log(1+\frac{|x-y|}{d_G(x)})=j_G(x,y),$$
because by a simple computation $\frac{t}{2} \le \log (1+t)$ for $t \in (0,1]$.

\end{proof}

Let us next consider convexity of metric balls.

\begin{lemma}
\label{ekalemma}
For all $\alpha\in(0,\arccos(1-2r^2))$ and $r\in(0,1/2]$ we have
\begin{equation*}
1+\cos\alpha-2\sqrt{2}\cos(\alpha/2)\sqrt{\cos\alpha+2r^2-1}\geq0.
\end{equation*}
\end{lemma}
\begin{proof}
Let us write the above inequality in the following form
\begin{equation*}
g(\alpha)=\cos\alpha-2\sqrt{2}\cos(\alpha/2)\sqrt{\cos\alpha+2r^2-1}\geq-1.
\end{equation*}
First we need to show that $g(\alpha)$ is strictly increasing. For this we need
\begin{equation}
g'(\alpha)=\sin\alpha\left(\frac{\sqrt{2}(r^2+\cos\alpha)}{\cos(\alpha/2)\sqrt{\cos\alpha+2r^2-1}}-1\right).
\end{equation}
In order to show that $g'(\alpha)>0$ we need to prove the following inequality
\begin{equation*}
h(\alpha)=\frac{\sqrt{2}(r^2+\cos\alpha)}{\cos(\alpha/2)\sqrt{\cos\alpha+2r^2-1}}>1.
\end{equation*}
Because
\begin{equation*}
h'(\alpha)=\frac{\sqrt{2}\tan(\alpha/2)(r^2-1)^2}{\cos(\alpha/2)(\cos\alpha+2r^2-1)^{3/2}}
\end{equation*}
it is easy to see that $h'(\alpha)>0$ for all $\alpha\in(0,\arccos(1-2r^2))$ and $r\in(0,1/2]$. Since $h(\alpha)$ is strictly increasing we see that $h(0)<h(\alpha)$ and with a simple computation we get
\begin{equation*}
h(\alpha)>h(0)=r+1/r\geq\frac{5}{2}>1.
\end{equation*}
Now we have shown that $g(\alpha)$ is strictly increasing. Therefore $g(\alpha)\geq g(0)=4r-1>-1$
and the assertion follows.

\end{proof}


\begin{theorem}
\label{nonconvextheorem}
Let $G = \mathbb{R}^n \setminus \{ z \}$, $z\in\mathbb{R}^n$, $x \in G$ and $r \in (0,1]$. The metric ball $B_s (x,r)$ is nonconvex if $r > 1/2$.
\end{theorem}
\begin{proof}
By symmetry it is sufficient to consider only the case $n=2$, $z=0$ and $x=2$. Let $r\in(0,1]$, $y\in\partial B_s(x,r)$ and denote $t=|y|$. Now $s_G(x,y)=r$ is equivalent to
\begin{equation*}
\left|2-y\right|=(2+t)r.
\end{equation*}
By the law of cosines we obtain $\left|2-y\right|^2=t^2+4-4t\cos\alpha$ which is equivalent to
\begin{equation*}
\left|2-y\right|=\sqrt{t^2+4-4t\cos\alpha}.
\end{equation*}

\begin{figure}[!ht]
\begin{center}
\label{kuva1}
\includegraphics[width=10cm]{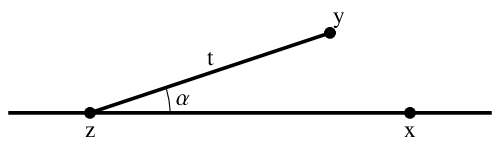}
\caption{Parametrization of $\partial B_s(x,r)$ in the proof of Theorem \ref{nonconvextheorem}.}
\end{center}
\end{figure}

Therefore we get
\begin{equation*}
\sqrt{t^2+4-4t\cos\alpha}=(2+t)r
\end{equation*}
and by solving for $t$ we obtain
\begin{equation}
\label{bothparts}
t^*(\alpha)=\frac{2\left(r^2+\cos\alpha\pm\sqrt{(1+\cos\alpha)(\cos\alpha+2r^2-1)}\right)}{1-r^2}.
\end{equation}
In this proof we select
\begin{equation*}
t(\alpha)=\frac{2\left(r^2+\cos\alpha-\sqrt{(1+\cos\alpha)(\cos\alpha+2r^2-1)}\right)}{1-r^2}.
\end{equation*}
Next we solve the following inequality to check, which values of $\alpha$ are interesting
\begin{equation*}
(1+\cos\alpha)(\cos\alpha+2r^2-1)>0.
\end{equation*}
With a simple substitution $a=\cos\alpha$ we get
\begin{equation*}
a^2+2r^2t+2r^2-1>0 \Leftrightarrow a>1-2r^2 \Leftrightarrow \alpha<\arccos(1-2r^2).
\end{equation*}
It is clear that if $r>1/2$, then $\arccos(1-2r^2)>\pi/3$. Therefore it is enough to focus on angles $\alpha\in(0,\pi/3)$.
The slope of the tangent of $\partial B_s(x,r)$ with respect to $\alpha$ according to the Definition \ref{slopedef} is
\begin{align*}
m(\alpha)&=\frac{t(\alpha)+\tan\alpha \ t'(\alpha)}{-t(\alpha)\tan\alpha +t'(\alpha)} \\ \\
&=\frac{\sqrt{(1+\cos\alpha)(\cos\alpha+2r^2-1)}+\sin\alpha\tan\alpha}{\sin\alpha-\sqrt{(1+\cos\alpha)(\cos\alpha+2r^2-1)}\tan\alpha},
\end{align*}
where
\begin{equation*}
t'(\alpha)=\frac{-2\left(r^2+\cos\alpha-\sqrt{(1+\cos\alpha)(\cos\alpha+2r^2-1)}\right)\sin\alpha}{(r^2-1)\sqrt{(1+\cos\alpha)(\cos\alpha+2r^2-1)}}.
\end{equation*}

Since $B_s(2,r)$ is symmetric with respect to $x_1$-axis, we need to show that $m(\alpha)<0$ for some $\alpha\in(0,\pi/3)$. It is clear that for all $\alpha\in(0,\pi/3)$ the inequality
\begin{equation*}
\sqrt{(1+\cos\alpha)(\cos\alpha+2r^2-1)}+\sin\alpha\tan\alpha>0
\end{equation*}
holds. Now we need to show that for some $\alpha\in(0,\pi/3)$ the inequality
\begin{equation*}
f(\alpha)=\sin\alpha-\sqrt{(1+\cos\alpha)(\cos\alpha+2r^2-1)}\tan\alpha<0
\end{equation*}
holds. By a simple computation we see that $f(0)=0$ and $f'(0)=1-2r<0$ for all $r>\frac{1}{2}$. Therefore $f(\alpha)<0$ for some sufficiently small $\alpha$. Now we have shown that $m(\alpha)<0$ for some $\alpha\in(0,\pi/3)$ and the assertion follows.
\end{proof}


\begin{theorem}\label{punctspace}
Let $G = \mathbb{R}^n \setminus \{ z \}$, $z\in\mathbb{R}^n$, $x \in G$ and $r > 0$. The metric ball $B_s (x,r)$ is (strictly) convex if $r \leq 1/2$ $(r < 1/2)$.
\end{theorem}
\begin{proof}
By symmetry it is sufficient to consider only the case $n=2$, $z=0$ and $x=2$. Let us use the following notation obtained in (\ref{bothparts})
\begin{align*}
t_1(\alpha)&=\frac{2\left(r^2+\cos\alpha-\sqrt{(1+\cos\alpha)(\cos\alpha+2r^2-1)}\right)}{1-r^2} \\
t_2(\alpha)&=\frac{2\left(r^2+\cos\alpha+\sqrt{(1+\cos\alpha)(\cos\alpha+2r^2-1)}\right)}{1-r^2}
\end{align*}
In the proof of Theorem \ref{nonconvextheorem} we showed that for $t_1(\alpha)$ the slope of the tangent of $\partial B_s(x,r)$ is
\begin{equation*}
m_1(\alpha)=\frac{\sqrt{(1+\cos\alpha)(\cos\alpha+2r^2-1)}+\sin\alpha\tan\alpha}{\sin\alpha-\sqrt{(1+\cos\alpha)(\cos\alpha+2r^2-1)}\tan\alpha}
\end{equation*}
and $m_1(0)<0$ if $r>1/2$. With this it is easy to see that
\begin{align*}
m_1(0)&>0 \ \ \text{if} \ \ r<1/2 ,\\
m_1(0)&\rightarrow\infty \ \ \text{if} \ \ r\rightarrow1/2.
\end{align*}
Similarly, for $t_2(\alpha)$ we write
\begin{equation*}
t_2'(\alpha)=\frac{2\left(r^2+\cos\alpha+\sqrt{(1+\cos\alpha)(\cos\alpha+2r^2-1)}\right)\sin\alpha}{(r^2-1)\sqrt{(1+\cos\alpha)(\cos\alpha+2r^2-1)}}
\end{equation*}
and
\begin{align*}
m_2(\alpha)&=\frac{t_2(\alpha)+\tan\alpha \ t'_2(\alpha)}{-t_2(\alpha)\tan\alpha +t'_2(\alpha)} \\ \\
&=\frac{-\sqrt{(1+\cos\alpha)(\cos\alpha+2r^2-1)}+\sin\alpha\tan\alpha}{\sin\alpha+\sqrt{(1+\cos\alpha)(\cos\alpha+2r^2-1)}\tan\alpha},
\end{align*}
where $m_2(\alpha)\rightarrow-\infty$ if $\alpha\rightarrow0$. To prove the theorem we need to show that $m_1'(\alpha)\leq0$ and $m_2'(\alpha)\geq0$ for all $\alpha\in(0,\arccos(1-2r^2))$ where $r\in(0,1/2]$. Firstly, we get
\begin{equation*}
m_1'(\alpha)=\frac{-2r^2\cos^2(\alpha/2)\frac{1}{\cos^2\alpha}\left(\varphi(\alpha)-2\sqrt{2}\sqrt{\cos^2(\alpha/2)\omega(\alpha)}\right)}{\sqrt{\varphi(\alpha)\omega(\alpha)}\left(\sin\alpha-\tan\alpha\sqrt{\varphi(\alpha)\omega(\alpha)}\right)^2},
\end{equation*}
where $\varphi(\alpha)=1+\cos\alpha$ and $\omega(\alpha)=\cos\alpha+2r^2-1$. Because
\begin{equation*}
\varphi(\alpha)\omega(\alpha)>0 \Leftrightarrow \alpha < \arccos(1-2r^2),
\end{equation*}
it is easy to see that for all $\alpha\in(0,\arccos(1-2r^2))$ and $r\in(0,1/2]$ the following inequality
\begin{equation*}
\sqrt{\varphi(\alpha)\omega(\alpha)}\left(\sin\alpha-\tan\alpha\sqrt{\varphi(\alpha)\omega(\alpha)}\right)^2>0
\end{equation*}
holds.
To prove that $m_1'(\alpha)\leq0$ we need to show that the inequality
\begin{equation*}
-2r^2\cos^2(\alpha/2)\frac{1}{\cos^2\alpha}\left(\varphi(\alpha)-2\sqrt{2}\sqrt{\cos^2(\alpha/2)\omega(\alpha)}\right)\leq0
\end{equation*}
holds. It is easy to see that
\begin{equation*}
-2r^2\cos^2(\alpha/2)\frac{1}{\cos^2\alpha}<0
\end{equation*}
for all $\alpha\in(0,\arccos(1-2r^2))$ and $r\in(0,1/2]$. In Lemma \ref{ekalemma} we showed that
\begin{equation*}
g(\alpha)=1+\cos\alpha-2\sqrt{2}\sqrt{\cos^2(\alpha/2)(\cos\alpha+2r^2-1)}\geq0.
\end{equation*}
Therefore $m_1'(\alpha)\leq0$. For $m_2(\alpha)$ we get
\begin{equation*}
m_2'(\alpha)=\frac{2r^2\cos^2(\alpha/2)\frac{1}{\cos^2\alpha}\left(\varphi(\alpha)+\sqrt{2}\sqrt{\cos(2\alpha)+4r^2\cos\alpha+4r^2-1}\right)}{\sqrt{\varphi(\alpha)\omega(\alpha)}\left(\sin\alpha+\tan\alpha\sqrt{\varphi(\alpha)\omega(\alpha)}\right)^2}.
\end{equation*}
Because
\begin{equation*}
\cos(2\alpha)+4r^2\cos\alpha+4r^2-1\geq0\Leftrightarrow\alpha<\arccos(1-2r^2),
\end{equation*}
we clearly see that $m_2'(\alpha)>0$ for all $\alpha\in(0,\arccos(1-2r^2))$ and the assertion follows.
\end{proof}

\begin{figure}[!ht]
\begin{center}
\label{kuva3}
\includegraphics[width=12cm]{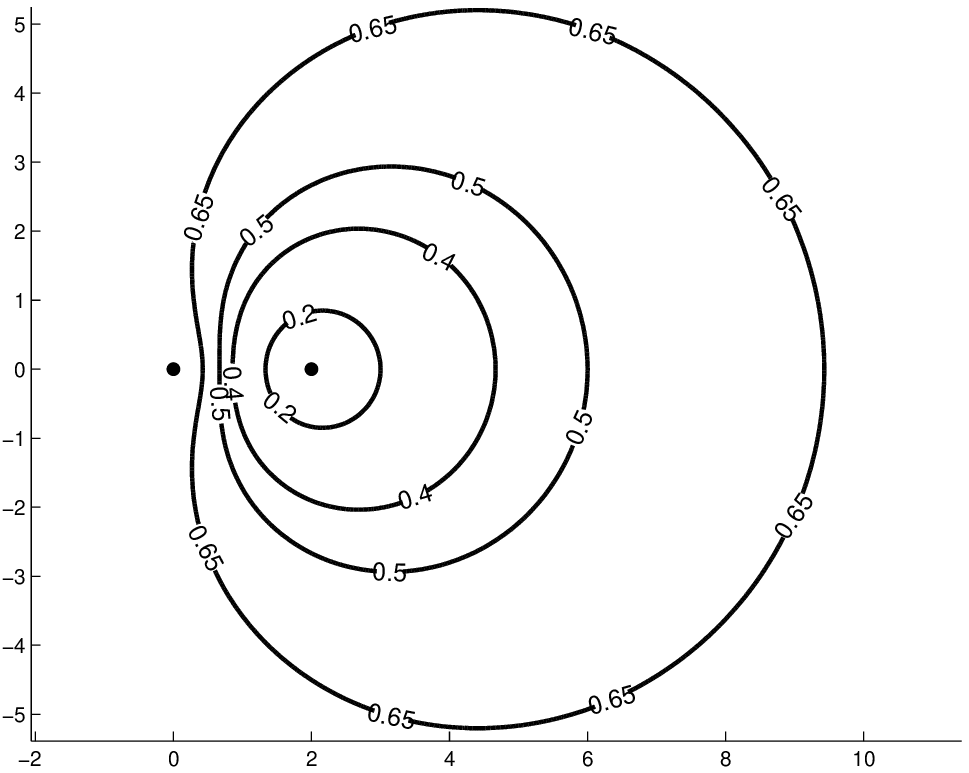}
\caption{Contours of the metric balls $B_s(x,r)$ in $\mathbb{R}^2 \setminus \{ 0 \}$ where $x=2$ and $r=0.2, 0.4, 0.5$ and $0.65$.}
\end{center}
\end{figure}


\section{Balls in half-space}

We consider the triangular ratio metric in half-space $\mathbb{H}^n$. We compare first $s$ with $\rho$ and then prove that the metric balls $B_s(x,r)$ are also Euclidean and thus always convex.

By the definition we obtain that for $x,y \in G = \mathbb{H}^n$
\begin{equation}\label{formulaHS}
  s_{G}(x,y) = \frac{|x-y|}{|x-y'|},
\end{equation}
where $y' = (y_1, \dots ,y_{n-1},-y_n) \notin G$. By Figure \ref{angmetfig9} it is clear that the supremum in \eqref{definiton} is attained at the point $z$.

\begin{figure}[!ht]
  \includegraphics[width=7cm]{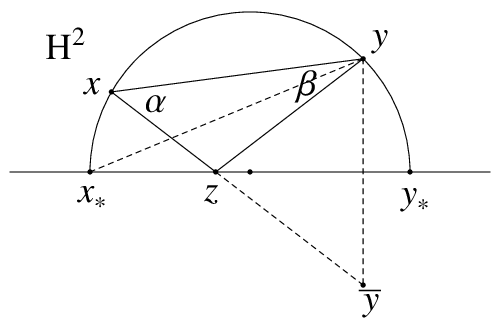}
  \caption{Formula for $s$ in $\Hn$ as in \eqref{formulaHS}.\label{angmetfig9}  }
\end{figure}

\begin{proposition}\label{Beardonprop}
  Then equality $s_{\Hn}(x,y) = \textnormal{tanh}\, \frac{\rho_{\Hn}(x,y)}{2}$ holds for all $x,y \in \Hn$.
\end{proposition}
\begin{proof}
  The assertion follows from \eqref{formulaHS} and \cite[7.2.1 (v)]{Beardon83}.
\end{proof}
\begin{proposition}\label{Beardonprop1}
 Balls in the triangular ratio metric are Euclidean balls.
\end{proposition}
\begin{proof} By \cite[(2.11)]{Vu2} hyperbolic balls in $\Hn$ are Euclidean balls and by Proposition \ref{Beardonprop} also balls in the triangular ratio metric are Euclidean balls.
\end{proof}
\begin{theorem}
\label{halfplanelemma}
  Let $x \in G = \mathbb{H}^n$ and $r\in(0,1)$. Then
  \[
    B_s(x,r)=B^n \left( x-e_n x_n \left( 1-\frac{1+r^2}{1-r^2} \right) ,\frac{2x_n r}{1-r^2} \right).
  \]
\end{theorem}

\begin{proof}
  By symmetry of the domain it suffices to consider only the case $n=2$ and we may assume that $x=(0,a)$, $a>0$. First we select points $q=(0,q_2)$ and $w=(0,w_2)$ such that $s(x,q)=s(x,w)=r$ where $q_2 \in (0,a)$, $w_2 > a$ and $r\in(0,1)$. By the definition of $s$ we get for $q_2$
\begin{equation*}
\frac{a-q_2}{a+q_2}=r
\end{equation*}
which is equivalent for $q_2=\frac{a-ar}{1+r}$. In a similar way we obtain $w_2=\frac{-a-ar}{-1-r}$. With a simple computation we get
\begin{align*}
|C|&=\frac{w_2+q_2}{2}=\frac{a(1+r^2)}{1-r^2} \\
R&=\frac{w_2-q_2}{2}=\frac{2ar}{1-r^2}
\end{align*}
and the assertion follows from Proposition \ref{Beardonprop1}.
\end{proof}


\begin{corollary}
\label{halfplanecor1}
  Let $x \in G = \mathbb{H}^n$ and $r\in(0,x_n)$. Then
  \[
    B^n(x,r)=B_s \left( x-e_n \left( x_n-\sqrt{x_n^2-r^2} \right) ,\frac{x_n-\sqrt{x_n^2-r^2}}{r} \right).
  \]
\end{corollary}

\begin{corollary}
\label{halfplanecor2}
  Let $x \in G = \mathbb{H}^n$ and $r > 0$. Then
  \[
    B^n \left( x, \frac{\sqrt{r^2+x_n^2}-x_n}{r} \right) =B_s \left( x-e_n \left( x_n-\sqrt{x_n^2+r^2} \right) ,r \right).
  \]
\end{corollary}

\begin{corollary}
  Let $x \in \mathbb{H}^n$ and $r \in (0,1)$. Then
  \begin{eqnarray*}
    & B_s(x,r)=B_\rho \left( x-e_n x_n \left( 1- a \right),t \right),\\
    & t = \arth\left(\frac{2r}{1+r^2}\right), \, b = \frac{(1+r^2)}{\cosh(t)(1-r^2)}.
  \end{eqnarray*}
\end{corollary}

\begin{proof}
  By symmetry of the domain it sufficies to prove the result in the case $n=2$ for $x=(0,a)$. By \cite[2.2]{vu85} and Theorem \ref{halfplanelemma}
\begin{equation*}
B_\rho(a i,t)=B^2(a\cosh(t)i,a\sinh(t))=B^2\left(\frac{a(1+r^2)}{1-r^2}i,\frac{2a r}{1-r^2}\right),
\end{equation*}
which is equivalent to
\begin{equation*}
\begin{cases}
a\cosh t=\frac{x_2(1+r^2)}{1-r^2} \\
a\sinh t=\frac{2x_2 r}{1-r^2}
\end{cases}
\end{equation*}
and the assertion follows.
\end{proof}

\section{Convexity of balls in punctured half-space and polygons}

We consider the triangular ratio metric in the punctured half-space $G= \Hn \setminus \{ e_n \}$. By \eqref{intersection}, Theorems \ref{punctspace} and \ref{halfplanelemma} it is clear that the following result holds.

\begin{lemma}\label{puncthalfplane}
  Let $x \in G= \Hn \setminus \{ e_n \}$ and $r \in (0,1/2]$. Then $B_s(x,r)$ is convex.
\end{lemma}

However, the upper bound for the radius $r$ in Lemma \ref{puncthalfplane} is not sharp. To see this we can choose $x$ close to $\partial \Hn$ and far from $e_n$. Now $B_s(x,r)$ is a Euclidean ball even for $r \in (1/2,r_0]$. For example for $x=e_n/10$ it can be verified that $B_s(x,r)$ is convex for $r \in (0,3/4]$, see Figure \ref{PHPpict}.

\begin{figure}[!ht]
\begin{center}
\label{kuva5}
\includegraphics[width=10cm]{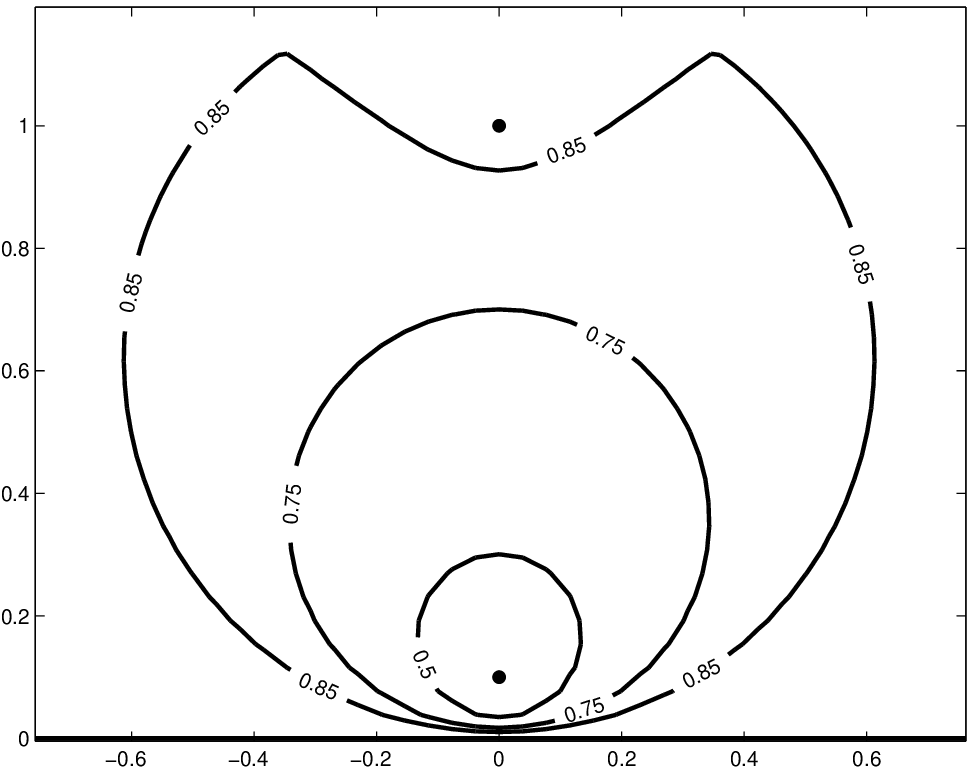}
\caption{Metric balls $B_s(x,r)$ in $\mathbb{H}^2\setminus \{ e_2 \}$ where $x=e_2/10$ and $r=0.5$, $0.75$ and $0.85$.\label{PHPpict}}
\end{center}
\end{figure}

The disks $B_s(x,r)$ for $x \in \R^2 \setminus \{ e_2 \}$ and large enough radius consists of two parts separated by curve $\{ (t,(t^2+1)/2) \colon t \in \R \}$. The lower part consists of $B_s(x,r)$ with respect to $\H^2$ and the upper part consists of $B_s(x,r)$ with respect to $\R^2 \setminus \{ e_2 \}$. The following theorem uses this idea and gives upper bound for the radius of convexity, when the center point $x$ is close to $\partial \HUP$.

\begin{theorem}\label{HUPthm}
  Let $x =(x_1,x_2) \in G= \HUP \setminus \{ e_n \}$ with $x_2 < |x_1|$ and $r \in (0,r_0]$, where
  \[
    r_0 = \frac{\sqrt{x_1^2+x_2^2}-\sqrt{2}x_2}{|x_1|+x_2}.
  \]
  Then $B_s(x,r)$ is convex.
\end{theorem}
\begin{proof}
  By simple computation we obtain that $(t^2+1)/2 > |t|$ and we show that $B_s(x,r_0)$ is below the curve $\{ (t,|t|) \colon t \in \R \}$. Let $R = 2x_2 r_0/(1-r_0^2)$ and $y = x_1+x_2(1+r_0^2)/(1-r_0^2)i$ be the Euclidean radius and center of $B_s(x,r)$. By geometry we obtain that $R=(x_1-y_2)/\sqrt{2}$ and thus
  \[
    \frac{2x_2 r_0}{1-r_0^2} = \frac{x_1-x_2 \frac{1+r_0^2}{1-r_0^2}}{\sqrt{2}}
  \]
  which implies the assertion.
\end{proof}


We consider next the triangular ratio metric in the angular domain
\[
  A_\alpha = \left\{ z \in \mathbb{R}^2 \colon \measuredangle (z,0,e_1) < \alpha/2 \right\}, \quad \alpha \in (0,2\pi).
\]
The boundary $\partial A_\alpha$ consists of two half-lines, which we call sides of the domain.

\begin{proposition}\label{ellipse}
  Let $x \in A_\alpha$, $\alpha \in (0,\pi]$ and $l$ be the line through the points 0 and $\overline{x}$. Then for all $y \in l \cap A_\alpha$ the maximal ellipse in $\overline{A_\alpha}$ touches both sides of the angular domain.
\end{proposition}
\begin{proof}
  If $\textrm{Im}\, x = 0$, then $\overline{x} = x$ and the assertion follows. 
  Let us denote the lines that contain $\partial A_\alpha$ by $s_1$ and $s_2$. Let $y \in A_\alpha$ and denote $y_1$ the reflection of $y$ with respect to line $s_1$ and similarily $y_2$ the reflection of $y$ with respect to line $s_2$. We consider maximal ellipses with foci $x$ and $y$ in the half-planes defined by lines $s_1$ and $s_2$. Formula \eqref{formulaHS} means geometrically that the maximal ellipse with foci $x$ and $y$ in half-plane defined by $s_1$ touches $s_1$ at the point $s_1 \cap [x,y_1]$. The same is true also for $s_2$. Note that the line containing $0$ and $x$ is the bisector of the lines through the origin and points $y_1$ and $y_2$. Now by geometry $|x-y_1| = |x-y_2|$ and thus the maximal ellipse in $A_\alpha$ touches both sides $s_1$ and $s_2$ and the assertion follows.
  \begin{figure}[!ht]
\begin{center}
\label{kuva6}
\includegraphics[width=7cm]{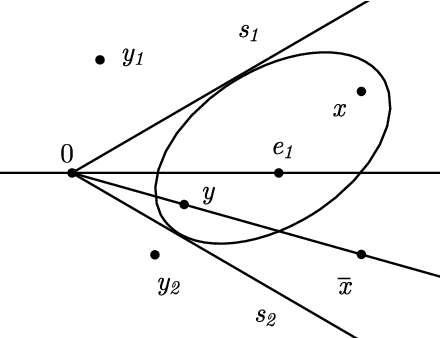}
\caption{Maximal ellipse in Proposition \ref{ellipse}.}
\end{center}
\end{figure}
\end{proof}

\begin{lemma}\label{alpha<pi}
  Let $G = A_\alpha$, $\alpha \in (0,\pi]$, $x \in G$ and $r \in (0,1)$. Then
  \[
    B_{s_G}(x,r) = B_{s_T}(x,r) \cap B_{s_U}(x,r),
  \]
  where $T$ and $U$ are the half-planes with $A_\alpha = T \cup U$. Moreover, $B_{s_G}(x,r)$ is always convex and $\partial B_{s_G}(x,r)$ is smooth for $r \le \sin \beta / \sin (\alpha/2)$, where $\beta = \measuredangle (x,0,e_1)$.
\end{lemma}
\begin{proof}
  By \eqref{intersection} we have $B_{s_G}(x,r) = B_{s_S}(x,r) \cap B_{s_T}(x,r)$ and by Theorem \ref{halfplanelemma}, $B_{s_G}(x,r)$ is convex as an intersection of two convex domains.

  Let us denote the line through $0$ and $\overline{x}$ by $l$. By Proposition \ref{ellipse} it is clear that $\partial B_{s_G}(x,r)$ is a circle if it does not intersect with $l$ at two distinct points and if $\partial B_{s_G}(x,r)$ is not a circle then it is not smooth. If $\partial B_{s_G}(x,r)$ is a circle and only touches $l$ then $\partial B_{s_G}(x,r) \cap l = \overline{x}$. We obtain that
  \[
    s_G(x,\overline{x}) = s_T(x,\overline{x}) = \frac{\sin \beta}{\sin \frac{\alpha}{2}}
  \]
  and the assertion follows.
\end{proof}

\begin{lemma}\label{alpha>pi}
  Let $G = A_\alpha$, $\alpha \in (\pi,2\pi)$ and $x \in G$. Then $B_s(x,r)$ is convex for $r \in (0,1/2]$ and the radius $1/2$ is sharp for $\beta = \measuredangle (x,0,e_1) < (\alpha-\pi)/2$.

  Moreover, if $\beta > (\alpha+\pi)/2$ then $\partial B_s(x,r)$ is smooth for
  \[
    r < \frac{\sin \left( \frac{\beta}{2}+\frac{\pi-\alpha}{4} \right) }{\sin \left( \frac{\pi+\alpha}{4}-\frac{\beta}{2} \right) }.
  \]
\end{lemma}
\begin{proof}
  The radius of convexity $1/2$ follows from \eqref{intersection} and Theorem \ref{punctspace}. Sharpness of the radius follows from the fact that if $\beta=\measuredangle (x,0,e_1) < (\alpha-\pi)/2$ then $s_G(x,y) = |x-y|/(|x|+|y|)$.

  If $\beta > (\alpha+\pi)/2$ then by the proof of Lemma \ref{alpha<pi}, $\partial B_s(x,r)$ is a circle if $r < \sin \beta' / \sin (\alpha'/2)$, where $b' = \beta/2+(\pi-\alpha)/4$ and $\alpha' = (\pi+\alpha)/2-\beta$.
\end{proof}

By combining the results in the angular domain we obtain the corresponding result in a polygon.

\begin{theorem}\label{polygonthm}
  Let $P \subset \mathbb{R}^2$ be a polygon and $x \in P$. Then $B_s(x,r)$ is convex for all $r \in (0,1/2]$. Moreover, if $P$ is convex then $B_s(x,r)$ is convex for all $r \in (0,1)$.
\end{theorem}
\begin{proof}
  The assertion follows from \eqref{intersection} and Lemmas \ref{alpha<pi} and \ref{alpha>pi}.
\end{proof}

\begin{proof}[Proof of Theorem \ref{mainthm1}]
  The assertion follows from Theorems \ref{punctspace}, \ref{halfplanelemma}, \ref{HUPthm} and \ref{polygonthm}.
\end{proof}

\begin{openproblem}
  Let $G \subsetneq \mathbb{R}^n$ be a convex domain and $x \in G$. Is $B_s(x,r)$ convex for all $r \in (0,1)$?
\end{openproblem}

\section{Inclusion relations of balls in general domains}

In this section and the following section we will consider the inclusion relations between metric balls in general domains and also some special domains.

\begin{theorem}\label{thm2-2}Suppose that $G\subset \mathbb{R}^n$ is a domain. For each $x\in G$ and $r\in (0,1)$, we have $$B^n \left( x, \frac{2r}{1+r}d_G(x) \right) \subset B_s(x, r)\subset B^n \left( x, \frac{2r}{1-r}d_G(x) \right).$$

\end{theorem}

\begin{proof}
We first show that $B^n(x, \frac{2r}{1+r}d_G(x))\subset B_s(x, r).$ For each $y\in B^n(x, \frac{2r}{1+r}d_G(x)),$ we have $$\frac{|x-y|}{\min_{z\in\partial G}\{|x-z|+|y-z|\}}\leq \frac{\frac{2r}{1+r}d_G(x)}{d_G(x)+d_G(y)}\leq \frac{r}{1+r}< r$$ when $d_G(y)\geq d_G(x)$, and
$$\frac{|x-y|}{\min_{z\in\partial G}\{|x-z|+|y-z|\}}\leq \frac{\frac{2r}{1+r}d_G(x)}{d_G(x)+d_G(x)-|x-y|}\leq  r$$ when $d_G(y)\leq d_G(x).$
Because $y$ is arbitrary,  we get $B(x, \frac{2r}{1+r}|x|)\subset B_s(x, r).$

For the second inclusion $B_s(x, r)\subset B^n(x, \frac{2r}{1-r}d_G(x)),$ let $y\in B_s(x, r)$ and assume that $d_G(x) \le d_G(y)$. Let $z\in \partial G$ be such that $|x-z|=d_D(x)$. Then $d_G(y)\leq|y-z|\leq \frac{1+r}{1-r}|x-z|=\frac{1+r}{1-r}d_G(x),$ which implies that $$|x-y|\leq r (|x-z|+|y-z|)\leq \frac{2r}{1-r}d_G(x).$$

\end{proof}

\begin{theorem}\label{thm2-1} Suppose that $G\subset \mathbb{R}^n$ is a domain. For each $x\in G$ and $r\in (0,1)$, we have $$B_j(x, \log(1+2r))\subset B_s(x, r)$$ and the  inclusion is sharp if there exists some points $w$ in $\partial B_s(x,r) $ such that $d_G(x)=d_G(w)$ and $$\partial G\cap S^{n-1}(w,d_G(w))\cap S^{n-1}(x,d_G(x))\neq \emptyset.$$ Moreover, for each $x\in G$ and $r\in (0,\frac{1}{3})$, we have  $$B_s(x, r)\subset B_j \left( x, \log(1+\frac{2r}{1-3r}) \right).$$

\end{theorem}

\begin{proof}  We first prove the first inclusion.
For given $x\in G$ and $y\in B_j(x,\log(1+2r))$, we have $|x-y|\leq 2r\min\{d_G(x),d_G(y)\}$. Then
$$\frac{|x-y|}{\min_{z\in \partial G}\{|x-z|+|y-z|\}}\leq \frac{|x-y|}{2d_D(x)}\leq r.$$

For the sharpness: if there exists a point $w$ in $\partial B_s(x,r) $ such that $d_G(x)=d_G(w)$ and $$\partial G\cap S(w,d_G(w))\cap S(x,d_G(x))\neq \emptyset,$$then let $z\in \partial G\cap S(w,d_G(w))\cap S(x,d_G(x))\neq \emptyset$. Hence $$|x-w|\leq r(|x-z|+|w-z|)=2r d_G(x),$$ which implies $j_G(x,w)=\log(1+2r).$

Next, we consider the second inclusion. By \cite[(3.9)]{Vu2}, we have $$B^n(x, r d_G(x))\subset B_k(x, \log\frac{1}{1-r}),$$ which together with the fact "$j_G(x,y)\leq k_G(x,y)$" show that $$B^n(x, r d_G(x))\subset B_k(x, \log\frac{1}{1-r})\subset B_j(x, \log\frac{1}{1-r}).$$  Hence, the second inclusion follows from Theorem \ref{thm2-2}.

\end{proof}

By Theorems \ref{thm2-2} and \ref{thm2-1} and \cite[(3.9)]{Vu2}, we get the following corollary.

\begin{corollary}
Suppose that $G\subset \mathbb{R}^n$ is a domain. For each $x\in G$ and $r\in (0,1)$, we have $$B_k(x, \log(1+2r))\subset B_s(x, r).$$ And for  $x\in G$ and $r\in (0,\frac{1}{3})$,  we have $$B_s(x, r)\subset B_k \left( x, \log(1+\frac{2r}{1-3r}) \right) .$$
\end{corollary}

\begin{figure}[!ht]
\centering
\includegraphics[width=0.4\textwidth,height=0.4\textwidth]{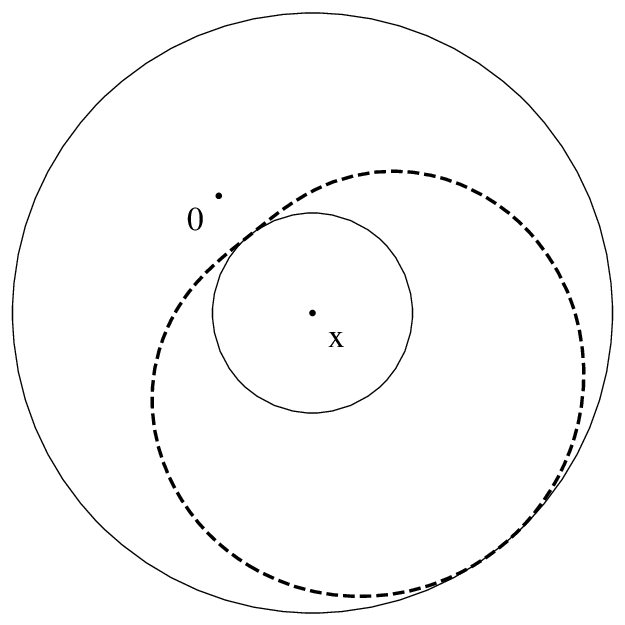}
\caption{This picture is in $\mathbb{R}^2\setminus\{0\}$ and the boundary of the $s$-metric disk in Theorem \ref{thm3-1} visualized as the dashed curve. Here $x=0.4-0.5I$ and $r=0.5$. \label{fig1}}
\end{figure}


\section{Inclusion relations of balls in some special domains}

First we consider the punctured spaces and get what follows.
\begin{theorem}\label{thm3-1}Let $G=\mathbb{R}^n\setminus \{0\}$. For each $x\in G$ and $r\in (0,1)$, we have $$B \left( x, \frac{2r}{1+r}|x| \right) \subset B_s(x, r)\subset B \left( x, \frac{2r}{1-r}|x| \right) .$$ Moreover, the inclusions are sharp (see Figure\eqref{fig1}).

\end{theorem}

\begin{proof} The inclusions follow from Theorem \ref{thm2-2}.
Sharpness of the first inclusion: choose $y=(1-\frac{2r}{1+r})x$. Obviously, $|x-y|=\frac{2r}{1+r}|x|$ and $\frac{|x-y|}{|x|+|y|}=r.$

Sharpness of the second inclusion: choose $y=(1+\frac{2r}{1-r})x$. Obviously, $|x-y|=\frac{2r}{1-r}|x|$ and $\frac{|x-y|}{|x|+|y|}=r.$

\end{proof}

\begin{theorem}\label{thm3-5}Let $G=\mathbb{R}^n\setminus \{0\}$. For each $x\in G$ and $r\in (0,1)$, we have $$B_j(x, m)\subset B_s(x, r)\subset B_j(x, M),$$ where $m=\log(1+2r)$, and $M=\log(1+\frac{2r}{1-r})$.  Moreover, the inclusions are sharp (see Figure \eqref{fig2}).

\end{theorem}

\begin{proof}The first inclusion follows from Theorem \ref{thm2-1}.

Sharpness: Choose $y\in \partial B_s(x,r)$ with $|y|=|x|$. Then $|x-y|=2r|x|$, whence $$\log \left( 1+\frac{|x-y|}{|x|} \right) =\log(1+2r).$$
\begin{figure}[!ht]
\includegraphics[width=0.4\textwidth,height=0.4\textwidth]{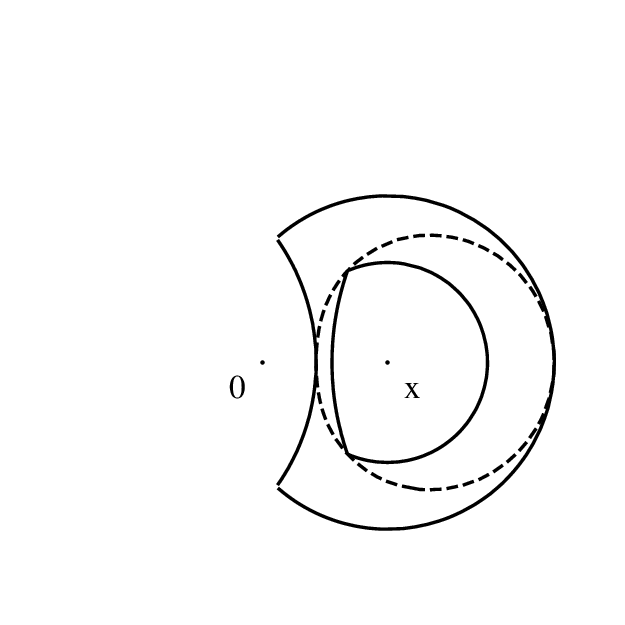}
\caption{This picture is in $\mathbb{R}^2\setminus\{0\}$ and the boundary of the $s$-metric disk in Theorem \ref{thm3-5} visualized as the dashed curve. Here $x=0.5$ and $r=0.4$. \label{fig2}}
\end{figure}
For the proof of the second inclusion, let $y\in B_s(x,r)$ with $|x|\leq |y|$.  Then $$|x-y|\leq r (|x|+|y|)$$ and $$|y|\leq \frac{1+r}{1-r}|x|.$$ Hence $$\log \left( 1+\frac{|x-y|}{|x|} \right) \leq \log \left( 1+\frac{2r}{1-r} \right) .$$

The case $y\in B_s(x,r)$ with $|x|\geq |y|$ follows from the fact that $s_G(x,y)$ and $j_G(x,y)$ are invariant under inversion about origin.

Sharpness: Choose $y=\frac{1+r}{1-r}x$. Then on one hand $\frac{|x-y|}{|x|+|y|}=r$, which implies $y\in \partial B_s(x,r)$.  On the other hand, $\log(1+\frac{|x-y|}{|x|})=\log(1+\frac{2r}{1-r})$ implies $y\in \partial B_j(x,r)$ which gives the sharpness.

\end{proof}

\begin{proof}[Proof of Theorem \ref{mainthm2}]
  The assertion follows from Theorems \ref{thm2-2}, \ref{thm2-1}, \ref{thm3-1} and \ref{thm3-5}.
\end{proof}

By Theorem \ref{thm3-1} and \cite[Theorem 3.3]{KV}, the following holds.

\begin{corollary}
Let $G=\mathbb{R}^n\setminus \{0\}$. (1) For each $x\in G$ and $r\in (0,1)$, we have $B_k(x, \log(1+2r))\subset B_s(x, r).$

(2)For each $x\in G$ and $r\in (0,\frac{1}{2})$, we have
$$B_s(x, r)\subset B_k(x, 2 \arcsin\frac{r}{1-r}).$$
\end{corollary}

\begin{conjecture}\label{con3-1}Let $G=\mathbb{R}^n\setminus \{0\}$. For each $x\in G$ and $r\in (0,1)$, we have $$B_k(x, m)\subset B_s(x, r)\subset B_k(x, M),$$ where $m=2 \arcsin r$, and $M=\log(1+\frac{2r}{1-r})$.  Moreover, the inclusions are sharp (see Figure \ref{fig3}).
\end{conjecture}
\begin{figure}[!ht]
\includegraphics[width=0.5\textwidth]{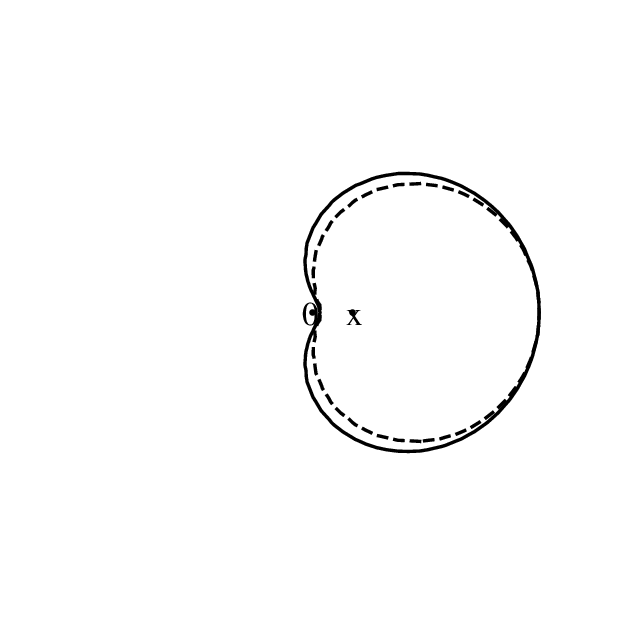}\hspace{5mm}
\includegraphics[width=0.4\textwidth]{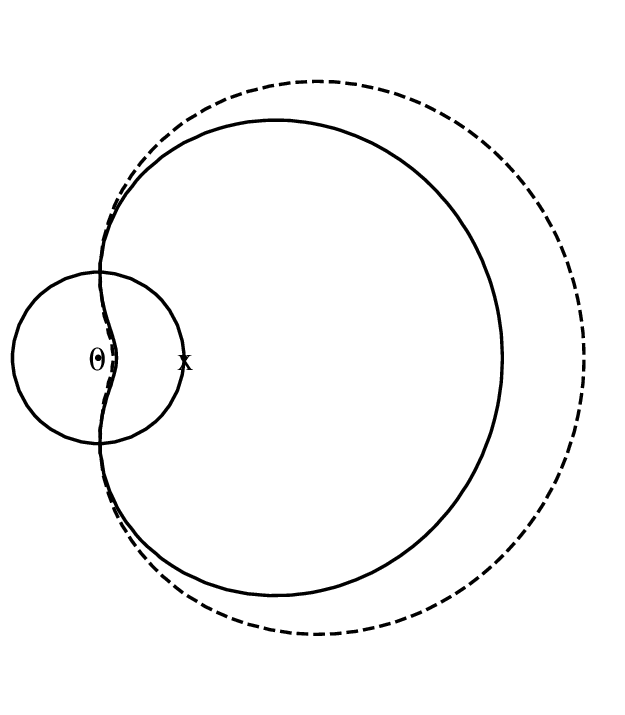}
\caption{This picture is in $\mathbb{R}^2\setminus\{0\}$ and the boundary of the $s$-metric disk in Conjecture \ref{con3-1} visualized as the dashed curve. In the left picture $x=2$, $r=0.7$, in the right one, $x=7$ and $r=0.7$. \label{fig3}}
\end{figure}

\noindent{\it Idea for the proof of Conjecture \ref{con3-1}.}  The first part. For $y\in B_k(x,m)$, let $\alpha=\angle(x0y)$. Then $\alpha\leq \sqrt{m^2-\log^2(\frac{|x|}{|y|})}$, and it suffices to prove
$$s_G(x,y)< \sin\frac{m}{2},$$
that is
$$\frac{|x-y|^2}{(|x|+|y|)^2}\leq  \frac{|x|^2+|y|^2-2|x||y|\cos\sqrt{m^2-\log^2(\frac{|x|}{|y|})}}{|x|^2+|y|^2+2|x||y|}\leq \sin^2\frac{m}{2} $$
and thus only need to prove the following:
$$f(t)= \left( 1-\sin^2\frac{m}{2} \right) + \left( 1-\sin^2\frac{m}{2} \right) t^2$$$$-2 \left( \sin^2\frac{m}{2} +\cos\sqrt{m^2-\log^2t} \right) t\leq 0,$$ where $m\in[0,\pi/2]$ and $t\in[0, e^m]$.

 The second part. For $y\in B_s(x,r)$, let $\alpha=\angle(x0y)$. Then \begin{eqnarray*}\alpha &=&\arccos \frac{|x|^2+|y|^2-|x-y|^2}{2|x||y|}\\
 &\leq& \arccos \frac{(1-r^2)|x|^2+(1-r^2)|y|^2-2r^2|x||y|}{2|x||y|}.\end{eqnarray*}

 It suffices to prove $$f(t)=\arccos^2 \left( \frac{1-r^2}{2}t+\frac{1-r^2}{2 t}-r^2 \right) +\log^2t\leq f \left( \frac{1+r}{1-r} \right)$$$$ =f \left( \frac{1-r}{1+r} \right) ,$$ where $r\in(0,1]$ and $t\in(\frac{1-r}{1+r},\frac{1+r}{1-r}).$

\begin{figure}[!ht]
\includegraphics[width=0.4\textwidth,height=0.4\textwidth]{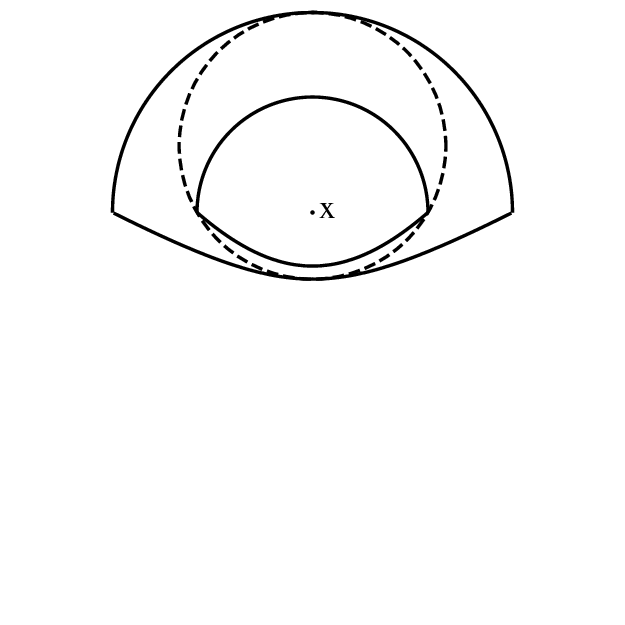}
\caption{This picture is in $\mathbb{H}^2$ and the boundary of the $s$-metric disk in Theorem \ref{thm3-6} visualized as the dashed curve. Here $x=2I$ and $r=0.5.$ \label{fig5}}
\end{figure}


For $x\in G=\mathbb{H}^n,$ by \cite[(2.11)]{Vu2} we have $B_k(x,r)= B^n(z, |x|\sinh r)$ with $z=|x|e_n \cosh r$, and we know that $B_s(x,r)=B^n(x-e_n x_n(1-\frac{1+r^2}{1-r^2}), \frac{2x_nr}{1-r^2})$, then it is easy to get the following Lemma.

\begin{lemma}\label{lem3-1} Let $x\in G=\mathbb{H}^n$ and $r\in(0,1)$. Then $$B_s(x,r)=B_k(x, \log(1+\frac{2r}{1-r})).$$
\end{lemma}

\begin{theorem}\label{thm3-6}Let $G=\mathbb{H}^n$. For each $x\in G$ and $r\in (0,1)$, we have $$B_j(x, m)\subset B_s(x, r)\subset B_j(x, M),$$ where $m=\log(1+\frac{2r}{\sqrt{1-r^2}})$, and $M=\log(1+\frac{2r}{1-r})$.  Moreover, the inclusions are sharp (see Figure \ref{fig5}).
\end{theorem}
\begin{proof}
It follows from Lemma \ref{lem3-1} and \cite[Theorem 4.2]{KV} and the fact that $B_k(x,r)\subset B_j(x,r)$.

\end{proof}

\begin{conjecture}Let $G=\mathbb{B}^n$. For each $x\in G$ and $r\in (0,1)$, we have $$B_j(x, m)\subset B_s(x, r)\subset B_j(x, M),$$ where $m=\log(1+2r)$, and $M=\log(1+\frac{2r}{1-r})$.  Moreover, the second inclusion is sharp.
\end{conjecture}


\noindent\textbf{Acknowledgement.} The research was supported by the Academy of Finland (project 209539).

\end{document}